\documentclass[EJP]{ejpecp} 

\usepackage{amsfonts,amssymb,amsmath,amsthm}
\usepackage{euscript,enumitem}
\usepackage{hyperref}
\usepackage{graphicx}
\usepackage[margin=10pt]{caption}

\AUTHORS{%
{Omer Angel}\footnote{ Department of Mathematics, University of British
  Columbia, Vancouver, BC V6T 1Z2, Canada, \EMAIL{angel@math.ubc.ca}}
\and {Vadim Gorin}\footnote{Department of Mathematics, Massachusetts Institute of Technology,
Cambridge, MA, 02139, USA and Institute for Information Transmission Problems, Moscow, 127994,
Russia, \EMAIL{vadicgor@gmail.com}} \and {Alexander~E.~Holroyd}\footnote{ Microsoft Research,
 Redmond, WA, 98052, USA, \EMAIL{holroyd@microsoft.com}}
     }

\SHORTTITLE{A pattern theorem for random sorting networks}

\TITLE{A pattern theorem for random sorting networks} 

\KEYWORDS{Sorting network; random sorting; reduced word; pattern; Young tableau} 

\AMSSUBJ{60C05; 05E10; 68P10} 

\SUBMITTED{October 6, 2011} 
\ACCEPTED{October 30, 2012} 

\VOLUME{17} \YEAR{2012} \PAPERNUM{99} \DOI{v17-2448}

\ABSTRACT{ A sorting network is a shortest path from $12\cdots n$ to $n\cdots 21$ in
  the Cayley graph of the symmetric group $S_n$ generated by
  nearest-neighbor swaps. A pattern is a sequence of swaps that forms an
  initial segment of some sorting network. We prove that in a uniformly
  random $n$-element sorting network, any fixed pattern occurs in at least
  $c n^2$ disjoint space-time locations, with probability tending to $1$
  exponentially fast as $n\to\infty$. Here $c$ is a positive constant which
  depends on the choice of pattern. As a consequence, the probability that
  the uniformly random sorting network is geometrically realizable tends to
  $0$.}

\newcommand{\eps}{\varepsilon}
\newcommand{\E}{\mathbb{E}}
\newcommand{\R}{\mathbb{R}}
\newcommand{\N}{\mathbb{N}}
\newcommand{\F}{\mathcal{F}}
\renewcommand{\P}{\mathbb{P}}
\newcommand{\eqd}{\stackrel{d}{=}}

\newcommand{\df}{\textbf}

\DeclareMathOperator{\Bin}{Bin}

\begin{document}



\section{Introduction}

Let $S_n$ be the group of all permutations $\sigma=(\sigma(1),\dots,\sigma(n))$ of $\{1,\dots,n\}$
with composition given by $(\sigma\tau)(i)=\sigma(\tau(i))$. We denote by $\sigma_j$ the adjacent
transposition or \df{swap} $(j\ j\!+\!1) = (1, \dots, j\!+\!1,j,\dots,n)$. A {\bf sorting network}
of \df{size} $n$ is a sequence $(s_1,s_2, \dots, s_N)$ of $N:=\binom{n}{2}$ integers with
$0<s_k<n$, such that the composition ${\sigma_{s_1}\sigma_{s_2} \cdots \sigma_{s_N}}$ equals the
reverse permutation ${(n,n-1,\dots,1)}$. We sometimes say that at time $k$ a swap occurs at
position $s_k$, and we illustrate a sorting network by a set of crosses with coordinates $(k,s_k)$
for $k=1,\ldots,N$. (This is natural, since the crosses may be joined by horizontal lines to give
a ``wiring diagram'' consisting of $n$ polygonal lines whose order is reversed as we move from
left to right; see Figure~\ref{fig:example}.)
\begin{figure}
  \begin{center}
{\scalebox{0.75}{\includegraphics{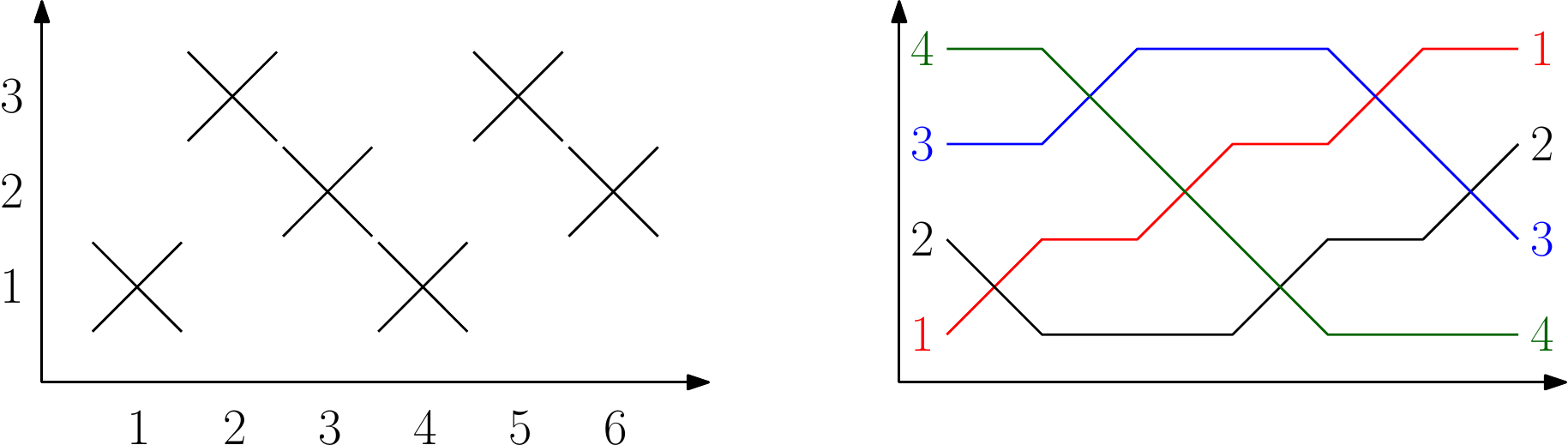}}}
    \caption{{\em Left:} the sorting network $(1,3,2,1,3,2)$ of size $4$,
      illustrated by crosses corresponding to its swaps.  {\em Right:} the
      associated wiring diagram.}
    \label{fig:example}
  \end{center}
\end{figure}

Interest in sorting networks was initiated by Stanley, who proved in \cite{St} that the number of
sorting networks of size $n$ is equal to the number of standard staircase-shape Young tableaux of
size $n$, i.e.\ those with shape $(n-1,n-2,\dots,1)$. Uniformly random sorting networks were
introduced and studied by Angel, Holroyd, Romik, and Virag in \cite{AHRV}, giving rise to many
striking results and conjectures.


A {\bf pattern} is any finite sequence of positive integers that is an initial segment of some
sorting network. Thus for example, $(1,2,1)$ and $(4,2)$ are patterns, but $(1,1)$ and $(1,2,1,2)$
are not. The \df{size} of a pattern is the minimum size of a sorting network that contains it as
an initial segment, which is also one more than the maximal element in the pattern.

Let $\omega=(s_1, \dots, s_N$) be a sorting network of size $n$ and let
$\gamma=(\gamma_1,\dots,\gamma_\ell)$ be a pattern. Let $[i,j]\subseteq[1,N]$ and $[a,b]\subseteq
[1,n-1]$, and consider the subsequence $t_1,\dots,t_\ell$ of $s_i,\dots,s_j$ consisting of
precisely those elements lying in the interval $[a,b]$. We say that the pattern $\gamma$
\df{occurs} at time interval $[i,j]$ and position $[a,b]$ (or simply at $[i,j]\times[a,b]$) if
$\gamma_u=t_u-a+1$ for $u=1,\dots,\ell$, and no $k\in[i,j]$ has $s_k\in\{a-1,b+1\}$. In other
words, the swaps in the space-time window $[i,j]\times[a,b]$ are precisely those of $\gamma$,
after an appropriate shift in location, and there are no swaps at the two adjacent positions,
$a-1$ and $b+1$, in this time interval. See Figure~\ref{fig:pattern1} for an example.
%
\begin{figure}[h]
  \begin{center}
  {\scalebox{0.87}{\includegraphics{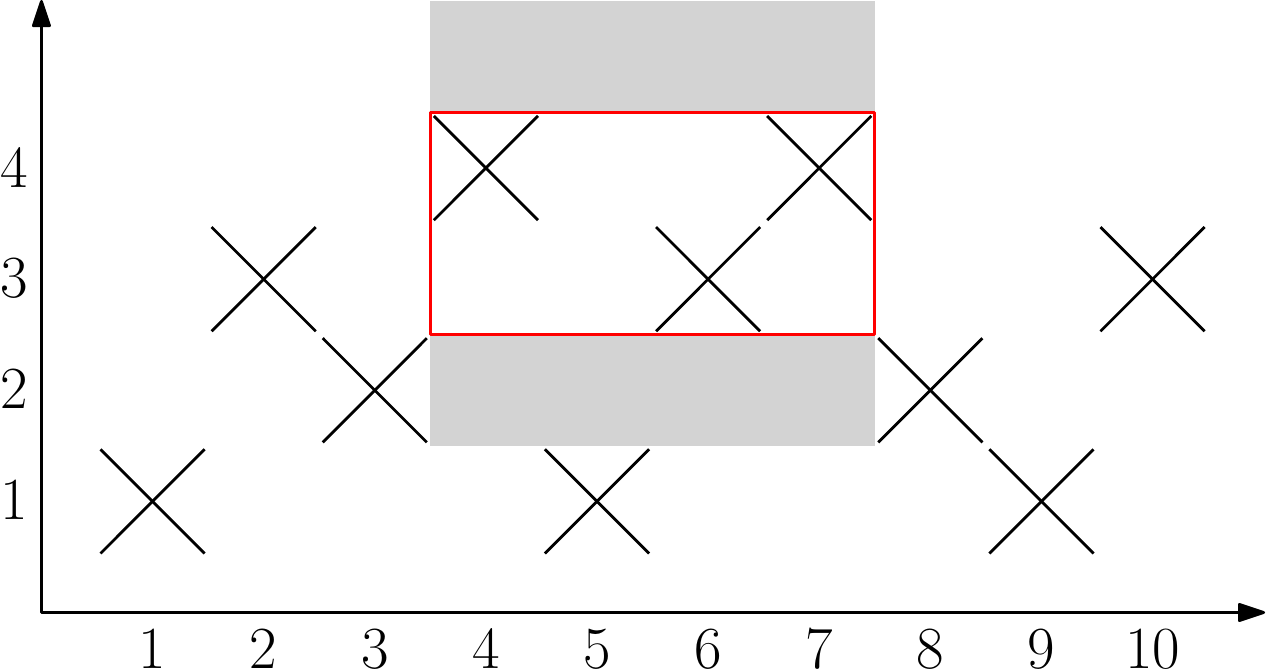}}}
    \caption{The pattern $(2,1,2)$ occurs in the sorting network
      $(1,3,2,4,1,3,4,2,1,3)$ at time interval $[i,j]=[4,7]$ and position
      $[a,b]=[3,4]$.  Note the requirement that the shaded regions contain
      no swaps.}
    \label{fig:pattern1}
  \end{center}
\end{figure}

We say that a pattern $\gamma$ \df{occurs $R$ times} in a sorting network $\omega$ if $R$ is the
maximum integer for which there exist pairwise disjoint rectangles
$\{[i_r,j_r]\times[a_r,b_r]\}_{r=1}^{R}$ such that $\gamma$ occurs at each.  See
Figure~\ref{fig:pattern2}.

\begin{figure}
\begin{center}
{\scalebox{0.87}{\includegraphics{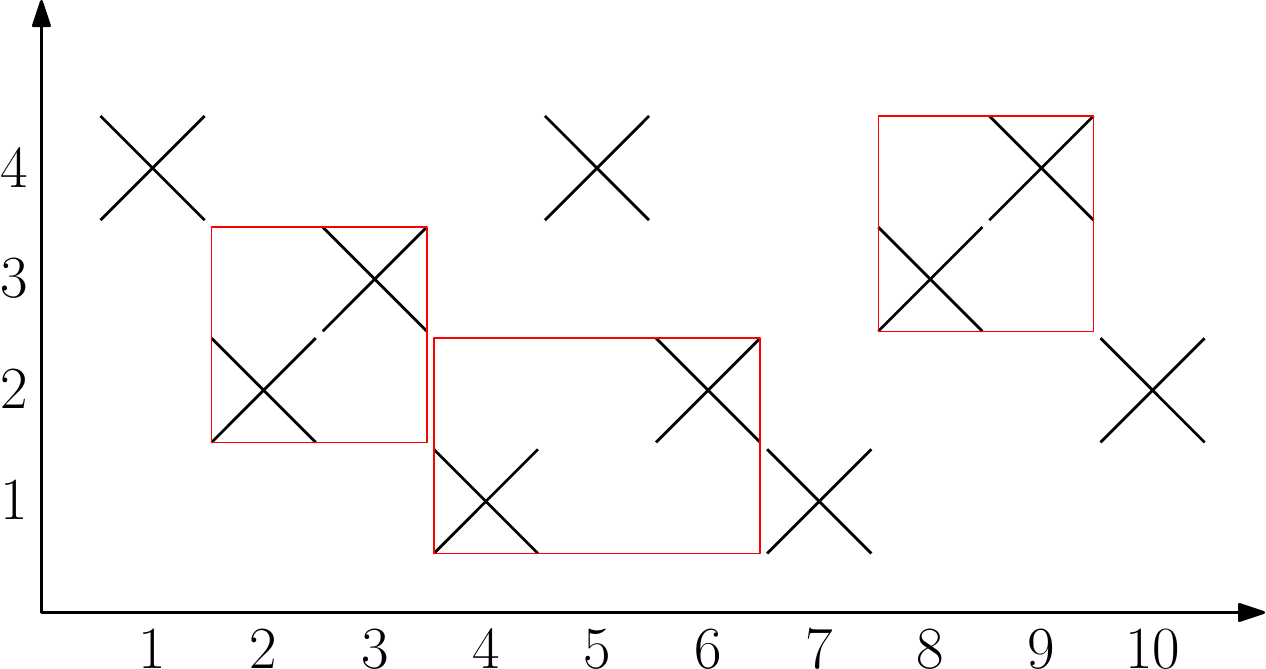}}}
   \caption{Pattern $(1,2)$ occurs  $3$ times in the sorting
      network $(4,2,3,1,4,2,1,3,4,2)$.}
    \label{fig:pattern2}
  \end{center}
\end{figure}

\begin{theorem}\label{any_subnetwork}
  Fix any pattern $\gamma$ of size $k$. There exist constants $c_1,c_2>0$
  (depending on $\gamma$) such that for every $n\ge k$, the
  pattern $\gamma$ occurs at least $c_1 n^2$
  times in a uniformly random sorting network of size $n$,
  with probability at least $1-e^{-c_2n}$.
\end{theorem}

We conjecture that the probability in Theorem~\ref{any_subnetwork} is in fact at least
${1-e^{-cn^2}}$ for some $c=c(\gamma)$.

We will prove Theorem~\ref{any_subnetwork} by establishing a closely related result about
uniformly random standard staircase-shape Young tableaux, and using a bijection due to Edelman and
Greene \cite{EG} between sorting networks and Young tableaux.

\sloppypar Write $\N=\{1,2,\ldots\}$.  A {\bf Young diagram} $\lambda$ is a set of the form
${\{(i,j)\in\mathbb{N}^2: 1\le j\le \lambda_i\}}$, where
$\lambda_1\ge\lambda_2\ge\lambda_3\ge\dots\ge 0$ are integers and $\sum_{i=1}^{\infty}
\lambda_i=:|\lambda|<\infty$. The numbers $\lambda_i$ are the {\bf row lengths} of $\lambda$. In
what follows we denote by $(\lambda_1,\lambda_2,\dots)$ the Young diagram with row lengths
$\lambda_1\ge\lambda_2\ge \dots$. We call an element $x=(i,j)\in\lambda$ a \df{box}, and draw it
as a unit square at location $(i,j)$ (with the traditional convention that $(1,1)$ is at the top
left and the first coordinate is vertical). A {\bf tableau} $T$ of shape $\lambda$ is a map from
$\lambda$ to the integers whose values are non-decreasing along rows and columns. We call $T(x)$
the \df{entry} assigned to box $x$. A \df{standard Young tableau} is a tableau $T$ of shape
$\lambda$ such that the set of entries  of $T$ is $\{1,2\dots,|\lambda|\}$.
 We are mostly
interested in {\bf standard
  staircase-shape Young tableaux} of size $n$, i.e.\ those with shape
staircase Young diagram $(n-1,n-2,\dots,1)$.

For $(i,j),(k,\ell)\in\mathbb N^2$ we write $(i,j)\leq (k,\ell)$ if $i\leq k$ and $j\leq \ell$.
For a Young diagram $\lambda$ and a box $(i,j)\in\lambda$, we define the \df{subdiagram}
$\lambda^{(i,j)}$ with top-left corner $(i,j)$ by $\lambda^{(i,j)}:=\{(k,\ell)\in
\lambda:(k,\ell)\ge (i,j)\}$; clearly $\lambda^{(i,j)}$ is mapped to a Young diagram by the
translation $(k,\ell)\mapsto (k-i+1,\ell-j+1)$.  If $T$ is a tableau of shape $\lambda$ then we
define the \df{subtableau} $T^{(i,j)}$ to be the restriction of $T$ to $\lambda^{(i,j)}$, and we
call $\lambda^{(i,j)}$ the \df{support} of $T^{(i,j)}$.

We say that two tableaux $S$ and $T$ of the same shape $\lambda$ are \df{identically ordered} if
for all $x,y\in\lambda$ we have $S(x)<S(y)$ if and only if $T(x)<T(y)$. Furthermore, if $S$ and
$T$ are tableaux or subtableaux, and there is a translation $\theta$ that  maps (bijectively) the
support of $S$ to the support of $T$, then we say that $S$ and $T$ are \df{identically ordered} if
for all $x,y$ in the support of $S$ we have $S(x)<S(y)$ if and only if
$T(\theta(x))<T(\theta(y))$. Figure~\ref{fig:subtableau} illustrates the above terminology.

\begin{figure}
  \begin{center}
    \includegraphics[width=0.6\textwidth]{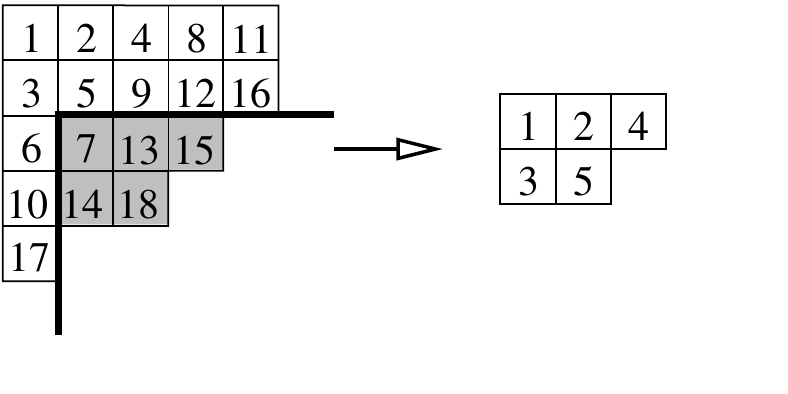}
    \caption{A standard Young tableau $T$ of shape $\lambda=(5,5,4,3,1)$,
      the subdiagram $\lambda^{(3,2)}$, the subtableau $T^{(3,2)}$, and a
      standard Young tableau identically ordered with $T^{(3,2)}$.  }
    \label{fig:subtableau}
  \end{center}
\end{figure}

Theorem~\ref{any_subnetwork} will be deduced from the following.

\begin{theorem}
  \label{th_any_corner} Let $T$ be any standard staircase-shape
  Young tableau of size $k$. For some positive constants $c_1'$, $c_2'$ and
  $c_3'$ (depending only on $k$), with probability at least $1-e^{-c'_3n}$, a
  uniformly random standard staircase-shape Young tableau of size $n\ge k$
  contains at least $c_1'n$ subtableaux with pairwise disjoint supports such
  that:
\begin{enumerate}
\item each is identically ordered with $T$;
\item all their entries are greater than $N-c'_2n$.
\end{enumerate}
\end{theorem}

As an application of Theorem~\ref{any_subnetwork} we prove that a uniformly random sorting network
is not geometrically realizable in the following sense. Consider a set $X$ of $n$ points in $\R^2$
such that no two points from $X$ lie on the same vertical line, no three points are collinear, and
no two pairs of points define parallel lines. Label the points $1,\dots,n$ from left to right
(i.e.\ in order of their first coordinate). Let $X_\phi$ be the set obtained by rotating $\R^2$ by
angle $\phi$ about the origin, and let $\sigma_\phi$ be the permutation found by reading the
labels in $X_\phi$ from left to right. As $\phi$ increases from $0$ to $\pi$, the permutation
$\sigma_\phi$ changes via a sequence of swaps, which form a sorting network. Any sorting network
that can be generated in this way is called {\bf geometrically realizable}. (Such networks were
called {\em
  stretchable} in \cite{AHRV}, but this term is used with a different
meaning in \cite{GR,GP}).

Goodman and Pollack \cite{GP} gave an example of a sorting network of size $5$ that is not
geometrically realizable. On the other hand, in \cite{AHRV}, it was conjectured (on the basis of
strong experimental and heuristic evidence) that a uniformly random sorting network is with high
probability {\em approximately} geometrically realizable, in the sense that its distance to some
random geometrically realizable network tends to zero in probability (in a certain natural
metric). The conjectures of \cite{AHRV} would also imply that, for fixed $m$, the sorting network
obtained by observing only $m$ randomly chosen particles from a uniformly random sorting network
of size $n\geq m$ is with high probability geometrically realizable as $n\to\infty$. (The
conjectures also imply that these size-$m$ networks have a limiting distribution as $n\to\infty$,
as well as providing a precise description of the limit. Certain aspects of the latter prediction
were verified rigorously in \cite{AH}.) However, we prove that with high probability a uniformly
random sorting network is {\em
  not} itself geometrically realizable.

\begin{theorem}
  \label{th_realizability} The probability that a uniformly random sorting
  network of size $n$ is geometrically realizable tends to zero as $n$
  tends to infinity.
\end{theorem}

While our proof yields an exponential (in $n$) bound on the probability that a uniform sorting
network of size $n$ is geometrically realizable, we believe the probability is in fact $O(e^{-c
n^2})$.

The paper is organized as follows. In Section~2 we recall basic definitions and the Edelman-Greene
bijection between sorting networks and standard Young tableaux. In Sections~3 and 4 we prove some
auxiliary lemmas about Young tableaux and sequences of random variables, respectively. In
Section~5 we prove Theorem~\ref{th_any_corner} and then deduce Theorem~\ref{any_subnetwork} as a
corollary. Finally, in Section~6 we prove Theorem~\ref{th_realizability}.

\section{Sorting networks and Young tableaux}

Edelman and Greene \cite{EG} introduced a bijection between sorting networks of size $n$ and
standard staircase-shape Young tableaux of size $n$, i.e.\ of shape $(n-1,n-2,\dots,1)$. We
describe it in a slightly modified version that is more convenient for us.

Given a standard staircase-shape Young tableaux $T$ of size $n$, we construct a sequence of
integers $s_1,\dots, s_N$ as follows. Set $T_{1}=T$ and repeat the following for $t=1,2,\dots,N$.

\begin{enumerate}
\item Let $x=(n-j,j)$ be the location of the maximal entry in the tableau $T_t$.
  Set $s_t=j$.
\item Compute the sliding path, which is a sequence $x_1,x_2,\dots, x_\ell$, such that $x_1=x$ and
 for $i=1,2,\dots$ we define $x_{i+1}$ to be the box among $\{x_i-(1,0),
x_i-(0,1)\}$ with larger entry in $T_t$, with the convention that $T_t(x)=0$ for every $x$ outside
the staircase Young diagram of size $n$. Let $\ell$ be the minimal $i$ such that $T_t(x_i)=0$.

\item Perform the sliding, i.e.\ define the tableau $T_{t+1}$ as follows.
    Set $T_{t+1}(x_{i})=T_t(x_{i+1})$ for $i=1,\dots,\ell-1$ and set
    $T_{t+1}(y)=T_t(y)$ for all boxes $y$ of the staircase Young diagram
    of size $n$ not belonging to $\{x_1,\dots,x_{\ell-1}\}$.
\end{enumerate}

An example of this procedure is shown in Figure~\ref{figure_sliding}. Edelman and Greene \cite{EG}
proved that the resulting sequence of numbers is indeed a sorting network, and furthermore that
the algorithm provides a bijection between standard staircase-shape Young tableaux and sorting
networks.

\begin{figure}
  \begin{center}
    \scalebox{0.95}{\includegraphics[width=\textwidth]{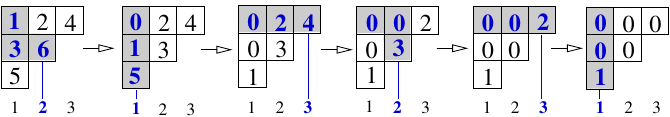}}
\caption{A standard staircase-shape Young tableau, sliding paths (shaded) and the sequence of
tableaux in the Edelman--Greene bijection. Here $n=4$ and the corresponding sorting network is
      $(2,1,3,2,3,1)$. Vertical lines show the correspondence between the positions of maximal entries
       in the tableaux and numbers $s_1,\ldots ,s_N$ of the sorting network.}
    \label{figure_sliding}
  \end{center}
\end{figure}

Now we fix $n$, consider the set of all sorting networks of this size and equip it with the
uniform measure. The Edelman--Greene bijection maps this measure to the uniform measure on the set
of all standard staircase-shape Young tableaux of size $n$.

Given a standard Young tableau $T$ of shape $\lambda$ with $|\lambda|=M$ we define a sequence of
Young diagrams by
\[
\lambda^i = \{x\in\lambda : T(x) \leq M-i\}.
\]
Thus $\lambda=\lambda^0 \supset \lambda^1 \supset \dots \supset \lambda^M=\emptyset$, and
$\lambda^i\setminus\lambda^{i+1}$ consists of the single box $T^{-1}(|\lambda|-i)$. If $T$ is a
uniformly random standard Young tableau of shape $\lambda$, then conditional on $\lambda^i,
\lambda^{i-1},\dots,\lambda^0$, the restriction of $T$ to $\lambda^i$ is uniformly random. Thus
the sequence of diagrams described above is a Markov chain.

\section{Some properties of Young tableaux}

In this section we present a fundamental result about Young diagrams (the hook formula) and deduce
some of its consequences.

When drawing pictures of Young diagrams we adopt the convention that the first coordinate $i$ (the
row index) increases downwards while the second coordinate $j$ (the column index) increases from
left to right. Given a Young diagram $\lambda$, its {\bf transposed diagram} $\lambda'$ is
obtained by reflecting $\lambda$ with respect to diagonal $i=j$. The column lengths of $\lambda$
are the row lengths of $\lambda'$.

For any box $x=(i,j)$ of a Young diagram $\lambda$, its {\bf arm} is the collection of
$\lambda_i-j$ boxes to its right: $\{(i,j')\in\lambda : j'>j\}$. The {\bf leg} of $x$ is the set
$\{(i',j)\in\lambda : i'>i\}$ of $\lambda'_j-i$ boxes below it. The union of the box $x$, its arm
and its leg is called the {\bf hook} of $x$. The number of boxes in the hook is called the {\bf
hook length} and is denoted by $h(x)$. The {\bf co-arm} is the set $\{(i,j')\in\lambda : j'<j\}$;
the {\bf co-leg} is the set $\{(i',j)\in\lambda : i'<i\}$, and their union (which does not include
$x$) is called the {\bf co-hook} and denoted by $\mathcal C(x)$. See Figure~\ref{fig:arm_leg}.
Finally, a {\bf corner} of a Young diagram $\lambda$ is a box $x\in\lambda$ such that $h(x)=1$, or
equivalently such that $\lambda\setminus \{x\}$ is also a Young diagram.
\begin{figure}\begin{center}
\scalebox{1.2}{\includegraphics{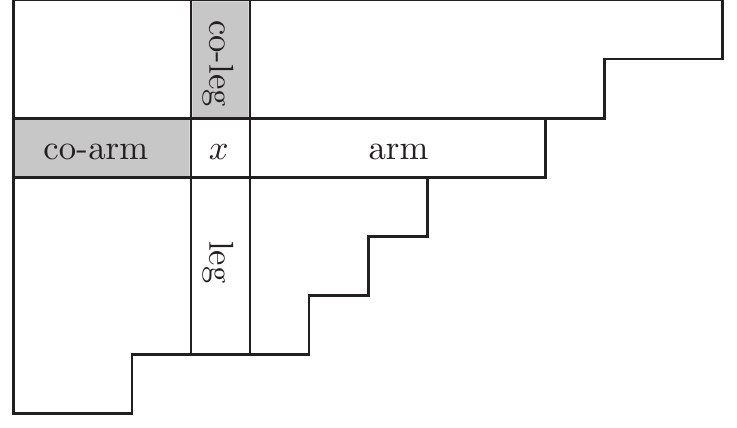}}
    \caption{Hook (clear) and co-hook (shaded) of a square $x$ in a Young
      diagram.}
    \label{fig:arm_leg}
  \end{center}\end{figure}

The dimension $\dim(\lambda)$ of a Young diagram $\lambda$ is defined as the number of standard
Young tableaux of shape $\lambda$ (thus named because it is the dimension of the corresponding
irreducible representations of the symmetric group).

\begin{lemma}[Hook formula; \cite{FRT}]
  \label{lemma_hook_formula} The dimension $\dim(\lambda)$ satisfies
  \[
  \dim(\lambda) = \frac{|\lambda|!}{\prod\limits_{x\in\lambda} h(x)}.
  \]
\end{lemma}

See e.g.\ \cite{FRT,M,Sagan} for proofs.


\begin{corollary}
  \label{corollary_distribution_max_entry}
  Let $T$ be a uniformly random standard Young tableau of shape $\lambda$,
  and let $x$ be a corner of $\lambda$. The location $T^{-1}(|\lambda|)$ of
  the largest entry is distributed as follows.
  \[
  \P\big(T^{-1}(|\lambda|) = x \big)
  = \frac{\dim(\lambda\setminus\{x\})}{\dim(\lambda)}
  = \frac{1}{|\lambda|} \prod_{z\in \mathcal C(x)}
  \frac{h(z)}{h(z)-1} .
  \]
(Note that $h(z)>1$ for any box in the co-hook $\mathcal C(x)$, so the right side is finite.)
\end{corollary}

\begin{proof}
This is immediate from Lemma~\ref{lemma_hook_formula}.
\end{proof}

\begin{lemma}\label{lemma_box_likely}
  Fix $\ell>0$. Let a Young diagram $\lambda$ be a subset of the staircase
  Young diagram of size $n$, and let $x=(i,j)$ be a corner of $\lambda$ with
  $i,j\ge n/4$ and $n-i-j\le \ell$. Let $T$ be a uniformly random standard
  Young tableau of shape $\lambda$. We have
  \[
  \P\bigl( T(x) = |\lambda| \bigr) \geq \frac{c}{n},
  \]
  where $c$ is a constant depending only on $\ell$.
\end{lemma}

There is nothing special about the bound $\frac{n}{4}$ on $i,j$ -- the lemma and proof hold as
long as $i,j\geq \eps n$, though the constant in the resulting bound tends to $0$ as $\eps \to 0$.

\begin{proof}[Proof of Lemma~\ref{lemma_box_likely}]
  The box $(i-k,j)$ of the co-hook has hook length $\lambda_{i-k}-j+k+1 \le
  n-i-j+2k+1 \le \ell+2k+1$. Similarly the box $(i,j-k)$ has hook length at
  most $\ell+2k+1$. It follows that
  \begin{align*}
    \P(T^{-1}(|\lambda|) = x)
     &= \frac{1}{|\lambda|} \prod_{k<i} \frac{h(k,j)}{h(k,j)-1}
    \prod_{k<j} \frac{h(i,k)}{h(i,k)-1} \\
     &\ge \frac{1}{n^2} \biggl( \prod_{k<n/4} \frac{\ell+2k+1}{\ell+2k}
     \biggr)^2.
  \end{align*}
  (Here we used that the factors are all decreasing in $h$, greater than
  $1$, and that $i,j\ge n/3$.) It is now easy to estimate
  \begin{align*}
  \biggl( \prod_{k<n/4} \frac{\ell+2k+1}{\ell+2k} \biggr)^2
  &\ge
  \biggl( \prod_{k<n/4} \frac{\ell+2k+1}{\ell+2k} \biggr)
  \biggl( \prod_{k<n/4} \frac{\ell+2k+2}{\ell+2k+1} \biggr) \\
  &= \frac{\ell + 2\lfloor n/4\rfloor + 2}{\ell+2}  > c n
  \end{align*}
  for some $c=c(\ell)$.
\end{proof}

\begin{lemma}
  \label{lemma_cond_prob}
  Let $T$ be a uniformly random standard Young tableau of shape $\lambda$,
  let $x$ and $y$ be two corners of $\lambda$ and $\ell = \|x-y\|_\infty$.
  Then
  \[
  \frac{\P(T^{-1}(|\lambda|) = x) }{ \P(T^{-1}(|\lambda|) = y) }
  \leq (\ell+1)(2\ell+1).
  \]
\end{lemma}

For our application all we need is a bound of the form $C(\ell)$ on this ratio, though we note
that the bound we get is close to optimal for a tableau of shape $(n+1,n,\dots,n)$ with $\ell+1$
rows, for large $n$.

\begin{proof}[Proof of Lemma~\ref{lemma_cond_prob}]
  To compare the expressions from
  Corollary~\ref{corollary_distribution_max_entry} for $x$ and $y$, let us
  introduce a partial matching between $\mathcal C(x)$ and $\mathcal C(y)$.
  We match boxes of the co-arm of $x$ and the co-arm of $y$ if they are in
  the same column. We match boxes of the co-leg of $x$ and the co-leg of
  $y$ if they are in the same row. All other boxes of $\mathcal C(x)$ and $\mathcal C(y)$ remain unmatched (see
  Figure~\ref{fig:matching}).

  \begin{figure}
\begin{center}
\scalebox{1.2}{\includegraphics{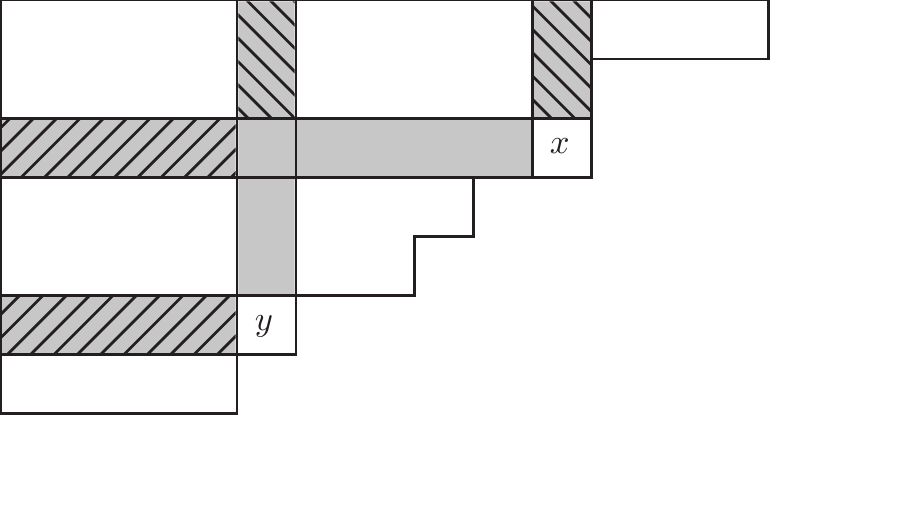}}
      \caption{The co-hooks $\mathcal C(x)$ and $\mathcal C(y)$ (shaded)
        and the matched parts of the co-arms and co-legs (hatched in
        matching directions).}
      \label{fig:matching}
    \end{center}\end{figure}

  Writing $x=(i_1,j_1)$ and $y=(i_2,j_2)$ without loss of generality assume
  that $i_1<i_2$ and $j_1>j_2$. Clearly, if $z\in\mathcal C(x)$ and
  $z'\in\mathcal C(y)$ are a matched pair, then $h(z')=h(z)\pm s$, where
  $s=i_2-i_1+j_1-j_2$ and the sign is plus if the box $z$ belongs to the
  co-leg of $x$ and minus otherwise. Let $M(x)$, $U(x)$ be the matched and
  unmatched parts of the co-hook $\mathcal C(x)$ and similarly for $y$. We have
  \begin{equation} \label{prod_formula}
    \frac{\P(T^{-1}(|\lambda|)=x)}{\P(T^{-1}(|\lambda|)=y)} =
    \dfrac{\prod\limits_{z\in U(x)} \dfrac{h(z)}{h(z)-1}}
    {\prod\limits_{z\in U(y)} \dfrac{h(z)}{h(z)-1}}
    \times \prod_{z\in M(x)} \dfrac{\left(\dfrac{h(z)}{h(z)-1}\right)}
    {\left(\dfrac{h(z)\pm s}{h(z)-1\pm s}\right)},
  \end{equation}
  where the choice of the sign $\pm$ depends on whether a box $z$ belongs to the co-arm
  or the co-leg of $x$.

  Let us bound the right side of \eqref{prod_formula}. First note that all
  the boxes in the co-leg of $x$ and all the boxes in the co-arm of $y$ are
  matched. The product over $z\in U(y)$ is at least $1$. Next, there are at
  most $\ell$ unmatched boxes of the co-arm of $x$ and their hook lengths
  are distinct. Consequently
  \[
  \prod_{z\in U(x)} \frac{h(z)}{h(z)-1} \le \prod_{m=2}^{\ell+1}\frac{m}{m-1}
  = \ell+1.
  \]

  Turning to the last product in \eqref{prod_formula}, a matched pair of
  boxes from the co-arms contributes to \eqref{prod_formula} the factor
  \[
  \dfrac{\left(\dfrac{h(z)}{h(z)-1}\right)}
  {\left(\dfrac{h(z)-s}{h(z)-1-s}\right)},
  \]
  which is easily seen to be less than $1$.

  Finally, every matched pair of boxes from the co-legs contributes to
  \eqref{prod_formula} the factor
  \[
  \dfrac{\left(\dfrac{h(z)}{h(z)-1}\right)}{\left(\dfrac{h(z)+s}{h(z)-1+s}\right)}
  = 1 + \frac{s}{(h(z)-1)(h(z)+s)},
  \]
  This is greater than $1$ for any $h(z)$. As $z$ varies over a co-leg of
  $x$, the values of $h(z)$ are distinct. Consequently, the contribution
  from the matched boxes from the co-legs is bounded from above by
  \[
  \prod_{m=2}^{\infty}
  \dfrac{\left(\dfrac{m}{m-1}\right)} {\left(\dfrac{m+s}{m-1+s}\right)}
  = \lim_{r\to\infty} \frac{r}{r+s}\;(s+1)
  = s+1 \le 2 \ell + 1.
  \]
  Multiplying all the aforementioned inequalities we get the required estimate.
\end{proof}

\section{Sequences of random variables}

Recall that a real-valued random variable $Y$ \df{stochastically dominates} another real-valued
random variable $Z$ if and only if there exist a probability space $\Omega$ and two random
variables $\widetilde Y, \widetilde Z$ defined on $\Omega$, such that $\widetilde Y \eqd Y$ and
$\widetilde Z \eqd Z$, and $\widetilde Y \ge \widetilde Z$ almost surely.

\begin{lemma}
  \label{lemma_n_tokens} Let $X_1,\dots,X_N$ be random variables taking values in
  $\{1,\dots,m,\infty\}$ such that a.s.\ each $a\in[1,m]$ appears exactly $r$ times.  Let $A_i$ be
  events, and define the filtration $\F_i = \sigma \left( X_1,\dots,X_i, A_1,\dots,A_{i-1}
  \right)$. Assume $\P(A_i\mid\F_i)\geq p$ a.s.\ for some $p>0$ and all $i$. Let $G_a$ be the
  event
  \[
  G_a=\bigcap_{i=1}^N \bigl( \{X_i\neq a\}\cup A_i\bigr) ,
  \]
  that is that $A_i$ occurs whenever $X_i=a$. Then $\sum_{a=1}^m 1_{G_a}$
  stochastically dominates the binomial random variable $\Bin(m,p^r)$.
\end{lemma}

To clarify the lemma, it helps to think of having $m$ counters initialized at 0. At each step
$i\le N$, a counter is selected dependent on $X_i$ (or no counter, signified by $X_i=\infty$), and
that counter is advanced (event $A_i$) with conditional probability at least $p$. The event $G_a$
is that the $a$th counter is advanced every time it is selected. Then after every counter has been
selected $r$ times, the number of counters with the highest possible value $r$ stochastically
dominates a $\Bin(m,p^r)$ random variable. Note that the order in which counters are selected may
depend arbitrarily on the past selections and advances. While this lemma seems intuitively clear
and perhaps even obvious, the precise assumptions on the dependencies among the events and
variables make the proof slightly delicate.

\begin{proof}[Proof of Lemma~\ref{lemma_n_tokens}]
  First, we want to extend the probability space, and define events $A'_i\subseteq A_i$ and a
  finer filtration $\F'_i$ in such a way that ${\P(A'_i\mid \F'_i) = p}$ for all $i$.

  Let $\Omega$ be our original probability space and let $\mu$ be our original probability
  measure. For $i=1,2,\dots ,N$ let $\mathcal E^i$ be the set of all elementary events in the
  finite $\sigma$--algebra $\F_i$ that have non-zero probabilities (with respect to $\mu$). The condition
  $\P(A_i\mid\mathcal F_i)\ge p$ means that $\P(A_i\mid E)\ge p$ for every $E\in\mathcal E^i$.
   For
  any $E\in\mathcal E^i$ let $\Omega^E_i$ denote the probability space $\{0,1\}$ with probability
  measure $\mu^E_i$ such that $\mu^E_i(1)={p}/{\P(A_i\mid E)}$. Our new probability space $\Omega'$ is
  the product of $\Omega$ and all $\Omega^E_i$:
  $$
   \Omega'=\Omega \times \prod_{i=1}^{N} \prod_{E\in\mathcal E^i} \Omega^E_i.
  $$
  In other words, an element of $\Omega'$ is a pair $(\omega,f)$, where $\omega\in\Omega$ and $f$
 is a function from $\bigsqcup_i \mathcal E^i$ to $\{0,1\}$ (here $\bigsqcup$
denotes set-theoretic disjoint union, so $\bigsqcup_i \mathcal E^i:=\cup_i\{(E,i):E\in
\mathcal{E}_i\}$). We equip $\Omega'$ with the probability
 measure $\mu'$ which is the direct product of $\mu$ and the measures $\mu^E_i$:
  $$
   \mu'=\mu \times \prod_{i=1}^{N} \prod_{E\in\mathcal E^i} \mu^E_i.
  $$
In what follows we do not distinguish between a random variable $X(\omega)$ defined on $\Omega$
and the random variable $X(\omega,f):=X(\omega)$ defined on $\Omega'$. In the same way we identify
any event $A$ of $\Omega$ with $\widetilde A:=\{(\omega,f)\in\Omega': \omega\in A\}\subseteq
\Omega'$. In what follows all the probabilities are understood with respect to $\mu'$.

For any $E\in\bigsqcup_i \mathcal E^i$ let $f^E$ denote the random variable on $\Omega'$ given by
$$
 f^E(\omega,f)=f(E).
$$

Now for any $E\in\mathcal E^i \subseteq \bigsqcup_j \mathcal E^j  $ set $$
B_i^E:=\{(\omega,f)\in\Omega'\mid \omega\in E,\, f(E)=1\}=E\cap\{f^E=1\}.
$$
Put it otherwise, $B_i^E$ is the event that both $E$ occurs and $f^E=1$. Denote
$$B_{(i)}=\bigcup_{E\in\mathcal E^i} B_i^E
$$ and let $A'_i=A_i\cap B_{(i)}$. Informally, to get $A'_i$ we cut $A_i$ into pieces $A_i\cap E$, replace
every such piece by $A_i\cap B_i^E$ and then glue pieces back together.

Let us introduce a filtration on $\Omega'$:
$$
 \F'_i=\sigma\Bigl(X_1,\dots,X_i,A_1,\dots,A_{i-1},\{f^E\}\Bigr),
$$
where $E$ runs over all elements of $\bigsqcup_{j=1}^{i-1} \mathcal E^j$.

Note that $A'_i\in\F'_{i+1}$. We claim that $\P(A'_i\mid \F'_i)=p$ for every $i$. Indeed, since
$A'_i$ is independent of all $f^E$ for $E\in\bigsqcup_{j=1}^{i-1} \mathcal E^i$, we have
$\P(A'_i\mid \F'_i)=\P(A'_i\mid \F_i)$. (Hear we mean that $\F_i$ is still  $\sigma \left(
X_1,\dots,X_i, A_1,\dots,A_{i-1}
  \right)$, although, now $\F_i$ lives in a different probability space.) But then, by the definition of $A'_i$,
  for every $E\in\F_i$ we have
  $$
   \P(A'_i\mid E)=\P(A_i\mid E) \frac{p}{\P(A_i\mid E)}=p.
  $$

  Moreover, consider any sequence of stopping times $1\le
  \tau_1<\dots<\tau_\ell\leq N$ (w.r.t.\ the filtration $\F'$). We claim
  that $\P\left(\bigcap_{i\leq \ell} A'_{\tau_i} \right) = p^\ell$.  The
  proof is a simple induction in $\ell$.  For $\ell=1$ we have
  \begin{align*}
    \P(A'_{\tau_1})&=\sum_{i=1}^N \P(A'_i\cap\{\tau_1=i\})\\
    &=\sum_{i=1}^N
    \P(\tau_1=i)\P(A'_i\mid \tau_1=i)=\sum_{i=1}^N
    \P(\tau_1=i)\cdot p = p,
  \end{align*}
  where in the last equality we used that $\P(A'_i\mid \F'_i)=p$ and $\{\tau_1=i\}\in\F'_i$. Now
assume that our statement is true for $\ell=h-1$. Then for $\ell=h$ we have
 \[
  \P\left(\bigcap_{i=1}^{h} A'_{\tau_i} \right)=\sum_{j=1}^N \P(\tau_1=j) \P(A'_j\mid \tau_1=j)
\P\left(\bigcap_{i=2}^h A'_{\tau_i}\mid A'_j\cap\{\tau_1=j\}\right).
 \]
 Note that for $i\ge 2$ the restriction of $\tau_i$ on the set $A'_j\cap\{\tau_1=j\}$ is again a
stopping time. Indeed, by the definition, $j< \tau_i\le N$ on $\{\tau_1=j\}$, and for $k>j$ we
have $\{\tau_i\le k\}\cap A'_j \cap \{\tau_1=j\}\in \F'_k$, since both $\{\tau_i\le k\}\in \F'_k$
and $A'_j \in \F'_k$ and $\{\tau_1=j\}\in \F'_k$. Therefore, using the induction assumption we
conclude that if $\P(A'_j\cap\{\tau_1=j\})>0$, then $\P(\bigcap_{i=2}^h A'_{\tau_i}\mid
A'_j\cap\{\tau_1=j\})=p^{h-1}$. Hence,
$$
  \P\left(\bigcap_{i=1}^{h} A'_{\tau_i} \right)=\sum_{j=1}^N \P(\tau_1=j) \P(A'_j\mid \tau_1=j)
p^{h-1}=\sum_{j=1}^N \P(\tau_1=j) p^h=p^h.
$$

  Now, let
  \[
  G'_a=\bigcap_{i=1}^N \bigl( \{X_i\neq a\}\cup A'_i\bigr)\subseteq G_a.
  \]
  Applying the above claim to the $r$ ordered stopping times $\tau_i$
  defined by
  \[
  \{\tau_1,\dots,\tau_r\} = \{k : X_k=a\}
  \]
  we find $\P(G'_a)=p^r$. Moreover, for any set $S\subseteq[1,m]$, by taking
  the $r|S|$ ordered stopping times $\tau^S_i$ defined by
  \[
  \{\tau_1,\dots,\tau_{r|S|}\} = \{k : X_k\in S\}
  \]
  we find
  \[
  \P\left(\bigcap_{a\in S} G'_a \right) =  p^{r|S|}.
  \]
  It follows that the events $G'_a$ are independent, and so
  \[
  \sum_{a=1}^m 1_{G_a} \geq \sum_{a=1}^m 1_{G'_a} \eqd \Bin(m,p^r). \qedhere
  \]
\end{proof}

\begin{lemma}\label{lemma_linear_time_estimate}
  Let $X_1,\dots,X_N$ be random variables taking values in $\{1,\dots,m,\infty\}$ such that a.s.\
  each $a\in[1,m]$ appears exactly $r$ times.  Denote $S_k(a):= \#\{i\leq k : X_i=a\}$, in particular $S_k(a)\le r$. Let
  $\widehat \F_k = \sigma(X_1,\dots,X_k)$, and suppose moreover that for some $c>0$ and all $a,k$,
  on the event $S_k(a)<r$ (which lies in $\widehat \F_k$), we have
  \[
  \P\bigl(X_{k+1}=a \mid \widehat \F_k\bigr) > \frac{c}{m}.
  \]
  Finally, let $D_k = \#\{a : S_k(a) = r\}$. Then for every $\eps>0$ there
  are constants $c_1,c_2$, depending on $c,r$ but not on $m$ or $N$, such that
  \[
  \P\big( D_{c_1 m} \le (1-\eps)m \big) < e^{-c_2 m}.
  \]
\end{lemma}

\begin{proof}
  Let $T_k = \sum_{a=1}^{m} S_k(a)$, clearly $0\le T_{k+1} -T_k\le 1$. Note that $T_k > m r - \eps m$
  implies ${D_k > (1-\eps)m}$. This is because $S_k(a)\le r$.

  On the event $D_k\leq (1-\eps)m$ there are at least $\varepsilon m$
  values $a$ for which $S_k(a)<r$, so by the condition of Lemma
  \ref{lemma_linear_time_estimate} we have $\E (T_{k+1}\mid \widehat
  \F_k)-T_k \geq c\eps$. Let $M_k$ be $c\eps k - T_k$, and let $M'_k$ be
  $M_k$ stopped when $D_k$ exceeds $(1-\eps)m$. More formally, the stopping
  time $K$ is the minimum number such that $D_K>(1-\eps)m$, and
  $M'_k = M_{k\wedge K}$.

  Observe that $M'_k$ is a supermartingale with bounded
  increments. Therefore, by the Azuma-Hoeffding inequality for
  supermartingales (which follows from the martingale version by Doob
  decomposition; see e.g.\ \cite{Az} or \cite[E14.2 and 12.11]{W}), for any
  $c_1>0$ there is a $c_2$ so that $\P(M_{c_1 m} > m) \leq e^{-c_2 m}$.

  If $M_{c_1 m} \leq m$ and $K>c_1 m$, then ${T_{c_1 m} \geq (c\eps c_1 m
    -1)m}$. If $c_1$ is such that $c\eps c_1 m - 1 > r$, this cannot hold,
  thus $M'$ is already stopped by time $c_1 m$ with probability at least
  $1-e^{-c_2m}$.
\end{proof}

\begin{corollary}\label{Corollary_n_tokens_with_estimate}
  Let $X_i$, $A_i$ for $i=1,\dots,N$ be two random sequences satisfying the
  assumptions of both Lemmas~\ref{lemma_n_tokens} and
  \ref{lemma_linear_time_estimate}. Let $\widehat G(a,i)$ be the
  intersection of the events $G_a$ and $\{S_i(a)=r\}$, i.e.\
  \[
  \widehat G(a,i) = \{S_i(a)=r\} \cap \bigcap_{j=1}^N \bigl( \{X_j\neq
  a\}\cup A_j\bigr).
  \]
  Set $\widehat Q(i)=\sum_a 1_{\widehat G(a,i)}$. There exist positive
  constants $c_1$, $c_2$, $c_3$ (which depend on $r$, $p$, $c$, but not on
  $m,N$) such that $\P\bigl(\widehat Q(c_1 m) > c_2 m\bigr) > 1-e^{-c_3 m}$.
\end{corollary}

If we again think about $m$ counters, then the corollary means simply that after time $c_1m$, with
probability at least $1-e^{-c_3m}$, at least $c_2m$ counters will have advanced $r$ times.

\begin{proof}[Proof of Corollary~\ref{Corollary_n_tokens_with_estimate}]
  Denote $ Q=\sum_a 1_{G_a}$. Lemma~\ref{lemma_n_tokens} implies that $Q$
  stochastically dominates a binomial random variable. Thus, by a standard
  large deviation estimate (see e.g.\ \cite[Chapter 27]{K}), for some
  positive constants $c_4$, $c_5$ we have
  \[
  \P(Q>c_4m) > 1-e^{-c_5m}.
  \]
  Take $\eps = c_4/2$ in Lemma~\ref{lemma_linear_time_estimate}. It follows that for some $c_1$
  with probability at least $1-e^{-c_6m}$ random variable $\widehat Q(c_1m)$ differs from $Q$ by
  not more than $c_4 m/2$. Thus,
\[
    \P\Bigl( \widehat Q(c_1m) > c_4m/2\Bigr)
    > 1-e^{-c_3m}.  \qedhere
\]
\end{proof}

\section{Proofs of the main results}

We are now ready to prove Theorems~\ref{any_subnetwork} and \ref{th_any_corner}. We denote by $S$
a fixed standard staircase-shape Young tableau of size $k$ and by $T$ a uniformly random standard
staircase-shape Young tableau of size $n$. In what follows $k$ and $S$ are fixed (and will
correspond to the pattern we are looking for) while $n$ tends to infinity. Given $S$, the idea is
to consider $c n$ specific disjointly supported subtableaux of $T$ in columns $\lfloor n/4
\rfloor, \dots, \lfloor 3n/4 \rfloor$ and show that linearly many (in $n$) of them are identically
ordered with $S$. Now we proceed to the detailed proofs.

\begin{proof}[Proof of Theorem~\ref{th_any_corner}]
  Within the staircase Young diagram $\lambda$ of size $n$ we fix $m :=
  \lfloor \frac{n-1}{2k-2} \rfloor$ disjoint subdiagrams $K_1,\dots,K_m$ of $\lambda$, each a
  translation of the staircase Young diagram of size $k$, placed along the border diagonal of $\lambda$ with no gaps
in-between. The total number of columns involved is
$$
  M:=\left\lfloor \frac{n-1}{2k-2} \right\rfloor (k-1),
$$
and we choose the column set $\lfloor n/4 \rfloor+1, \dots, \lfloor n/4 \rfloor +M$. An example is
shown  in Figure~\ref{Fig:2disj}.

\begin{figure}
\begin{center}
\includegraphics[width=0.5\textwidth]{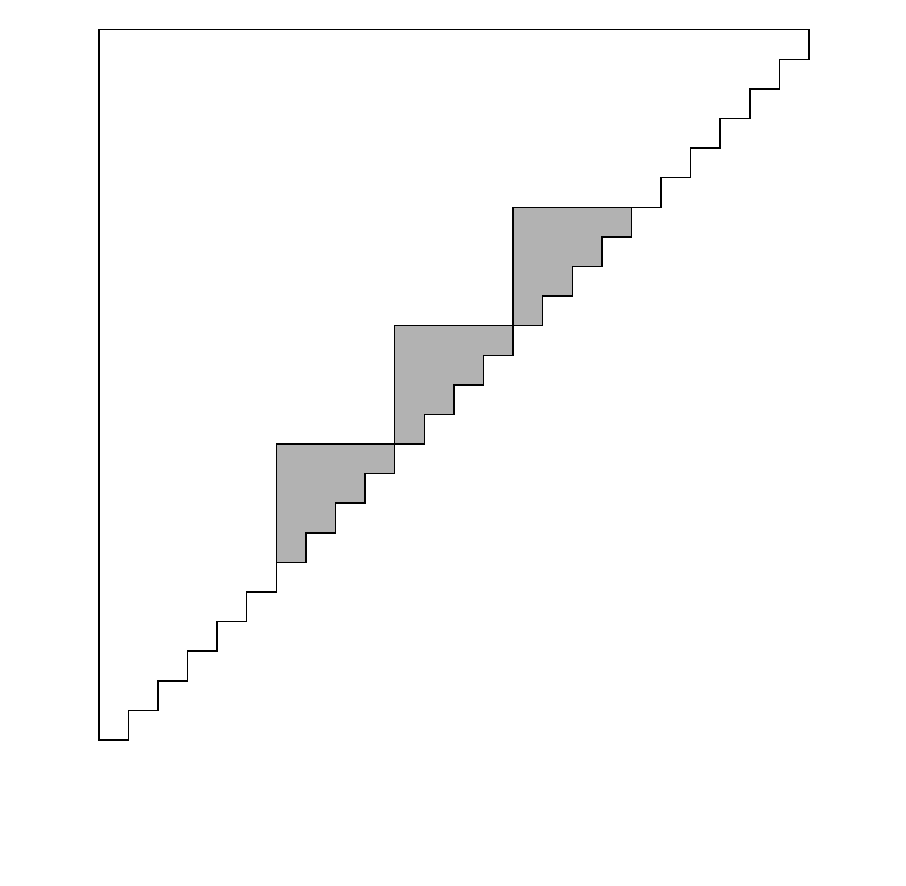}
\end{center}
  \caption{Three disjoint subdiagrams $K_1,K_2,K_3$ of a staircase Young
    diagram. Here $n=25$ and $k=5$.}
  \label{Fig:2disj}
\end{figure}

  Let $N := \binom{n}{2}$ and $r := \binom{k}{2}$. We now construct
  sequences $X_t$ and $A_t$ ($1\le t\le N$) to which we shall apply
  Lemmas~\ref{lemma_n_tokens} and \ref{lemma_linear_time_estimate}, as
  random variables on the probability space of standard staircase-shape
  Young tableaux $T$ of size $n$ with uniform measure. Set $X_t=a$ if
  $T^{-1}(N+1-t)$ belongs to $K_a$ and set $X_t=\infty$ if $T^{-1}(N+1-t)$
  does not belong to $\bigcup_a K_a$. Note that each $a\in\{1,\dots,m\}$
  appears exactly $r$ times among $X_1,\dots,X_N$.

  \newenvironment{mylist}{\begin{list}{(\roman{mycount})}
      {\usecounter{mycount}\itemsep 0pt}}{\end{list}}

  Next, we define the events $A_t$. Let $a=X_t$, and suppose $X_t$ is the $i$th
    occurrence of $a$ among $X_1,\dots,X_t$. The event $A_t$ occurs if and
    only if at least one of the following holds.
  \begin{enumerate}
  \item $a = \infty$.
  \item The box $T^{-1}(N-t+1)$ is in the same relative position within
      $K_a$ as $S^{-1}(r-i+1)$ is within a staircase Young diagram of
      size $k$.
  \item $A_s$ does not occur for some $s<t$ for which $X_s=a$.
  \end{enumerate}
  In other words, $A_t$ fails to occur precisely if for some number $a$
  the locations of entries $\{N-t+1,\dots,N\}$ imply that the subtableau
  supported by $K_a$ and $S$ are not identically ordered, and $A_s$ occurs for all $s<t$ (for that $a$).

  Let us also phrase this in terms of counters. Recall that a uniformly
  random standard staircase-shape Young tableau $T$ is associated with a
  Markov chain of decreasing Young diagrams $\{\lambda^t\}$. Each step of
  this Markov chain is a removal of a box from a Young diagram. If the box
  $x=\lambda^t \setminus \lambda^{t-1}$ removed at step $t$ belongs to
  $K_a$, then we choose the $a$th counter at this step. The counter
  advances if either the position of $x$ is the correct one for $K_a$ and
  $S$ to be identically ordered, or if the correct order of the entries of $T$
  inside $K_a$ was already broken at an earlier step. Clearly, if the $a$th
  counter advances $r$ times, then the subtableau of $T$ with support $K_a$
  is identically ordered with $S$.

  We shall see that the sequences $X_t$ and $A_t$, and the numbers $r$,
  $m$, $N$, satisfy the conditions of Lemma~\ref{lemma_n_tokens} and
  \ref{lemma_linear_time_estimate} with
$$p=\frac{1}{2k^3}.$$
  Theorem~\ref{th_any_corner} then follows immediately by applying
  Corollary~\ref{Corollary_n_tokens_with_estimate} for sequences $X_t$
  and $A_t$.

  As already noted, every $a \in \{1,\dots,m\}$ appears among
  $X_1,\dots,X_{N}$ exactly $r$ times.  Thus it remains to bound from below
  the conditional probabilities of $A_t$. Let $\mathcal F_t$ be as in
  Lemma~\ref{lemma_n_tokens}. We must prove that $\P(A_t\mid\mathcal F_t)\geq
  p$. Let $\mathcal G_t$ be the larger $\sigma$-algebra generated by
  $\lambda^0,\dots,\lambda^{t-1}$ together with $X_t$.

  If $X_t=\infty$ then $A_t$ occurs, and there is nothing to prove. So
  suppose $X_t=a\neq\infty$.  Now, on $X_t=a$, and given $\lambda^{t-1}=\mu$,
  there are at most $k-1$ corners of $\mu$ in $K_a$ (which
  correspond to possible positions of the box ${T^{-1}(N-t+1)}$).
  Lemma~\ref{lemma_cond_prob} implies that the probabilities of any two of
  these possibilities have a ratio of at most $2k^2$ (since the parameter
  $\ell$ in that lemma is at most $k-2$). Thus, for any Young diagram
  $\nu$ obtained from $\mu$ by removing a box inside $K_a$ we have
  \[
  \P(\lambda^t=\nu \mid \lambda^{t-1}=\mu,\, X_t=a) \geq p.
  \]
  Now, the Markov property of the sequence $\lambda^t$ imply that the same
  bound holds conditioned on all of $\lambda^0,\dots,\lambda^{t-1}$, i.e.\
  we have
   \[
  \P(\lambda^t=\nu \mid \mathcal G_t) \geq p
  \]
   on the event $\{X_t=a,\, \lambda^{t-1}=\mu\}$. Therefore, also
   \[
  \P(\lambda^t=\nu \mid \mathcal F_t) \geq p.
  \]

  Coming back to the bound on conditional probability of $A_t$, if some
  previous $A_s$ with $s<t$ and $X_s=X_t$ did not occur then $A_t$ occurs
  and $\P(A_t\mid \mathcal F_t)=1\ge p$. Otherwise, occurance of $A_t$
  depends on the position of the box ${T^{-1}(N-t+1)}$; specifically, $A_t$
  occurs if this box is the correct one according to $S$ of the possible
  boxes in the subdiagram $K_a$. We have shown above that each of the
  possible positions of this box has conditional probability at least
  $p$. Since exactly one of the positions corresponds to the
  event $A_t$, we conclude that
  \[
  \P(A_t\mid \mathcal F_t) \geq p.
  \]

  Finally, let us check that the sequence $X_t$ satisfies the conditions of
  Lemma~\ref{lemma_linear_time_estimate}.  Observe that the condition
  $S_t(a)<r$ means that the subdiagram $K_a$ is not completely filled with
  entries greater than $N-t$.  Thus, $S_t(a)<r$ if and only if
  $\lambda^{t}\cap K_a\neq\emptyset$, which is equivalent to $\lambda^t$
  having at least one corner in $K_a$.  Applying
  Lemma~\ref{lemma_box_likely} for $\lambda^t$ and this corner yields that
  for some positive constant $c$, on the event $S_t(a)<r$,
  \[
  \P(X_{t+1}=a \mid \lambda^{t}=\mu) > \frac{c}{m}
  \]
  Now, the Markov property of the sequence $\lambda^t$ imply that the same
  bound holds conditioned on all of $\lambda^0,\dots,\lambda^t$, and
  therefore also conditioned on the coarser $\sigma$-algebra $\widehat \F_t$.
\end{proof}

We now deduce Theorem~\ref{any_subnetwork} using the Edelman-Greene bijection.

\begin{proposition}
  \label{Proposition_any_subnetwork_simplified} Fix any pattern $\gamma$ of size $k$.
  There exist constants $c_3$, $c_4$ and $c_5$ (depending on $\gamma$) such
  that for every $n\ge k$, the pattern $\gamma$ occurs at least $c_3n$ times within the time interval $[1,c_4n]$ of a uniformly random sorting network of size $n$ with probability at least $1-e^{-c_5n}$.
\end{proposition}

Note that Proposition~ \ref{Proposition_any_subnetwork_simplified} differs from
Theorem~\ref{any_subnetwork} in that we consider only the beginning of the network and hence only
find a linear number of occurrences of $\gamma$.

\begin{proof}[Proof of Proposition~\ref{Proposition_any_subnetwork_simplified}]
  Clearly, it suffices to prove
  Proposition~\ref{Proposition_any_subnetwork_simplified} for patterns of
  length $k(k-1)/2$, or in other words a sorting network of size $k$.  Such a
  pattern $\gamma=(\gamma_1,\dots,\gamma_{k(k-1)/2})$ corresponds via the
  Edelman-Greene bijection to some standard staircase-shape Young tableau
  $T_\gamma$ of size $k$.  Consider a larger standard staircase-shape Young
  tableau $S$ of size $k+2$, which is a padded version of $T_\gamma$: entries of the hook of $(1,1)$
  are the numbers $1,\dots,2k+1$ (in an arbitrary admissible order) and the
  remaining staircase-shaped Young tableau of size $k-1$ contains
  $2k+2,\dots,(k+2)(k+1)/2$ and is identically ordered with $T_\gamma$.
  An example of this construction is shown in Figure~\ref{Figure_padding}.


  \begin{figure}
    \begin{center}
      \scalebox{0.8}{\includegraphics{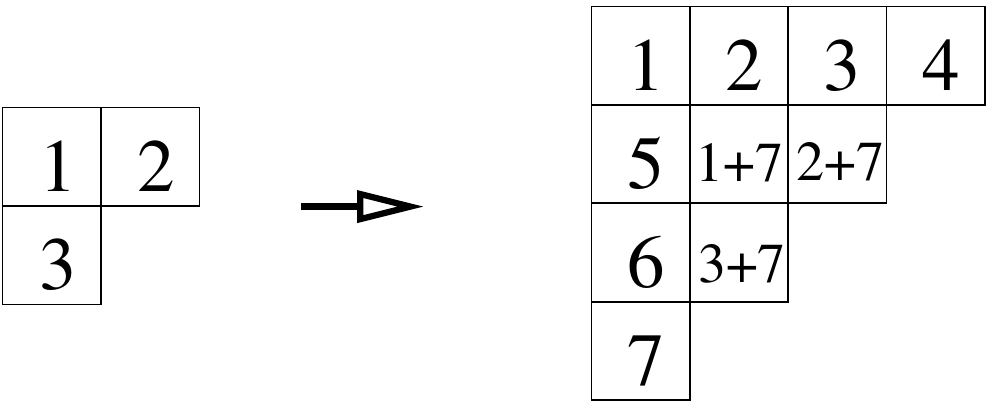}}
      \caption{``Padding'' a tableau $T_\gamma$ to get $S$. Here $k=3$.}
      \label{Figure_padding}
    \end{center}
  \end{figure}

  Let $c_3$, $c_4$ and $c_5$ be the constants $c_1'$, $c_2'$ and $c_3'$ of Theorem
  \ref{th_any_corner}, respectively. Let $T$ be a standard staircase-shape Young tableau of size $n$ having at
  least $c_3n$ disjointly supported subtableaux  identically ordered
  with $S$, furthermore, all the entries of these
  subtableaux are greater than $N-c_4n$.
  (Theorem~\ref{th_any_corner} implies that a uniformly random standard
  staircase-shape Young tableau of size $n\ge k$ is of this kind with
  probability at least $1-e^{-c_5n}$.) Suppose that the support of the $\ell$th such
  subtableau ($\ell=1,2,\dots,c_3n$) is a subdiagram $K_\ell$ with top-left corner
   $( n-j_{\ell}-k, j_{\ell})$. Let
  $K'_\ell$ denote the subdiagram with top-left corner $(n-j_{\ell}-k+1,j_{\ell}+1)$ and note that the subtableau
  with support $K'_\ell$ is identically ordered with $T_\gamma$.

  Let $\omega$ be the sorting network corresponding to $T$ via the
  Edelman-Greene bijection. Note that in the Edelman-Greene bijection,
  every tableau entry moves towards the boundary of the staircase Young
  diagram until it becomes the maximal entry in the tableau, and then it
  disappears and adds to the sorting network a swap in position $j$, where
  $j$ is the column of the entry just before it disappeared. It follows
  that all the entries starting in $K_\ell$ disappear in the columns
  $j_\ell,\dots,j_\ell+k$ and, thus, add to the sorting network swaps $s_i$
  satisfying $j_\ell\leq s_i \leq j_\ell+k$.
  Furthermore, observe that all the entries starting in $K'_\ell$ disappear (in
  columns $s_i$ satisfying $j_\ell < s_i < j_\ell+k$) before the entries in
  $K_\ell\setminus K'_\ell$. Finally, note that until all entries starting in $K'_\ell$ disappeared
  no other entry can disappear in columns $j_\ell,\dots,j_\ell+k$.

  We conclude that for every $\ell$, the pattern $\gamma$ occurs in $\omega$
at $[1,t_\ell]\times[j_\ell+1,j_\ell+k-1]$.  Thus, pattern $\gamma$ occurs in $\omega$ at least
$c_3n$ times within
  the time interval $[1,c_4n]$.
\end{proof}

\begin{proof} [Proof of Theorem~\ref{any_subnetwork}]
  Let $c_3$, $c_4$, $c_5$ be the constants from Proposition~\ref{Proposition_any_subnetwork_simplified}, and let $m:=\lceil c_4
  n\rceil$. For $t=1,\dots, \lfloor N/m \rfloor$ let $I_t$ be
  the set of all sorting networks $\omega$ of size $n$ such that $\gamma$
  occurs in $\omega$ at least $c_3n$ times within the time interval $[(t-1)m+1, t m]$.
  Proposition~\ref{Proposition_any_subnetwork_simplified} yields that $\P(I_1)\ge
  1-e^{-c_5n}$.

  A uniformly random sorting network $(s_1,s_2\dots,s_{N})$ is
  stationary in the sense that $(s_1,\dots,s_{N-1})$ and
  $(s_2,\dots,s_{N})$ have the same distributions (see \cite[Theorem~1]{AHRV}).
  Thus $\P(I_t)$ does not depend on $t$.

  There exist constants $c_6>0$ and $n_0$ such that if $n>n_0$, then $\left\lfloor N/m\right\rfloor e^{-c_5n} \le e^{-c_6n}$.
  Let $c_1=\min(c_3/(4c_4),c_3/n_0)$ and $c_2=\min(c_5,c_6)$.
  Let
  $I$ denote the set of all sorting networks $\omega$ of size $n$ such that
  $\gamma$ occurs $c_1 n^2$ times in $\omega$. If $n>n_0$ then we have
  \[
  \P(I) \ge \P\Bigl(\bigcap_{t} I_t\Bigr)
  \geq 1 - \sum_{t} \bigl(1-\P(I_t)\bigr)
  \geq 1 - \left\lfloor \frac{N}{m}\right\rfloor e^{-c_5n}\ge 1-e^{-c_2n}.
  \]
  And if $k\le n \le n_0$, then $I_1\subseteq I$ and
  \[
  \P(I)\ge \P(I_1) \ge 1-e^{-c_5n} \ge 1-e^{-c_2n}.
  \qedhere\]
\end{proof}

\section{Uniform sorting networks are not geometrically realizable}

\begin{proof}[Proof of Theorem~\ref{th_realizability}]
  Goodman and Pollack proved in the paper \cite{GP} that there exists a
  sorting network $\gamma$ of size $5$ that is not geometrically
  realizable. This sorting network is shown in\ Figure~\ref{fig:non_geom}.
  (This is the smallest possible size of such a network.)

  \begin{figure}[h]
\begin{center}
\scalebox{0.87}{\includegraphics{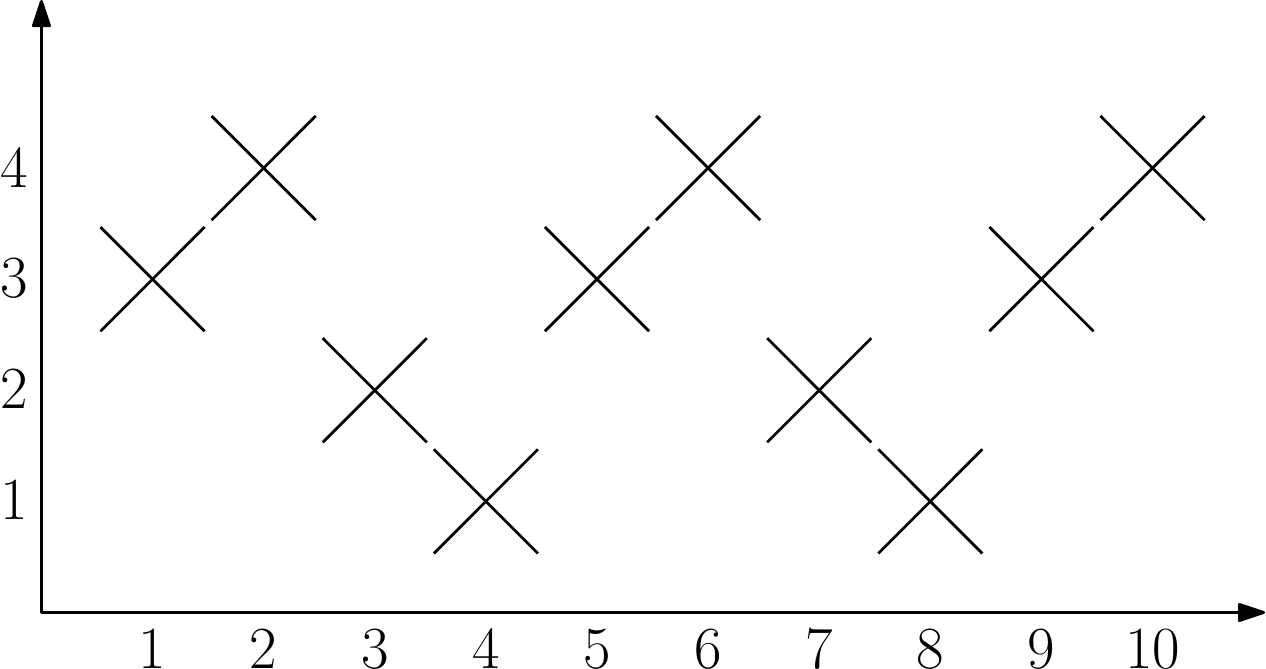}}
      \caption{A sorting network that is not geometrically realizable.}
      \label{fig:non_geom}
    \end{center}\end{figure}

  Let us view $\gamma$ as a pattern. Suppose that $\gamma$ occurs in a
  sorting network $\omega$ at time interval $[1,t]$ and position $[a,b]$.
  We claim that $w$ is not geometrically realizable. Indeed,
  if $\omega$ were a geometrically realizable sorting networks associated
  with points $x_1,\ldots ,x_n\in\R^2$ (labeled from left to right), then $\gamma$
  would be a geometrically realizable sorting network associated with the
  points $x_a,\dots,x_b$.

  Proposition~\ref{Proposition_any_subnetwork_simplified} yields that with
   tending to $1$ probability $\gamma$ occurs within the time interval $[1,c_4n]$
   of a uniformly random sorting network $\omega$ of size $n$ and
  thus $\omega$ is not geometrically realizable.
\end{proof}


\ACKNO{ We thank the anonymous referees for valuable comments. O.A. has been supported by the
University of Toronto, NSERC and the Sloan Foundation. V.G. has been supported by Microsoft
Research, Moebius Foundation for Young Scientists, ``Dynasty'' foundation, RFBR-CNRS grant
10-01-93114, the program ``Development of the scientific potential of the higher school'' and by
IUM-Simons foundation scholarship. }


\end{document}